\pdfoutput=1
\documentclass[11pt,a4paper,oneside]{article}
\usepackage[utf8]{inputenc} 
\usepackage[english]{babel} 

\usepackage{lmodern}
\usepackage[T1]{fontenc}
\usepackage{amsmath,mathtools,amsfonts,amssymb} 
\usepackage{slashed} 
\usepackage{mathrsfs}
\usepackage[babel=true]{microtype}
\usepackage[autostyle=true]{csquotes} 

\usepackage{comment}

\usepackage[a4paper,marginparwidth=2cm,margin=1in]{geometry}

\usepackage{amsthm}
\usepackage{thmtools}
\usepackage[pdfusetitle]{hyperref}
\usepackage[citestyle=numeric-comp,
            bibstyle=numeric-comp,
            backend=biber,
            url=false,
            doi=true,
            isbn=false,
            giveninits=true,
            maxcitenames=100,
            maxbibnames=100,
            sorting=nyt]{biblatex}
\AtEveryBibitem{\clearfield{month}}
\AtEveryBibitem{\clearfield{day}}
\AtEveryBibitem{%
  \ifentrytype{thesis}
    {}
    {
      \ifentrytype{online}{}
      {
      \clearfield{url}%
      \clearfield{urldate}%
      }
    }%
}

\usepackage{tikz}
\usetikzlibrary{cd} 

\usepackage[shortlabels]{enumitem}
\newlist{myenumi}{enumerate}{1}
\setlist[myenumi,1]{label=\upshape(\roman*)}
\newlist{myenuma}{enumerate}{1}
\setlist[myenuma,1]{label=\upshape(\alph*)}
\usepackage{color}
\usepackage{todonotes}
\declaretheorem[name=Theorem, numberwithin=section]{thm}
\declaretheorem[name=Theorem, numbered=no]{thm*}
\declaretheorem[name=Lemma,numberlike=thm]{lem}
\declaretheorem[name=Lemma,numbered=no]{lem*}
\declaretheorem[name=Corollary,numberlike=thm]{cor}
\declaretheorem[name=Proposition,numberlike=thm]{prop}
\declaretheorem[name=Definition,numberlike=thm, style=definition]{defi}

\declaretheorem[name=Conjecture,numberlike=thm, style=remark]{conj}

\declaretheorem[name=Example, numberlike=thm, style=remark]{ex}
\declaretheorem[name=Remark, numberlike=thm, style=remark]{rem}

\declaretheorem[name=Theorem]{thmx}

\declaretheorem[name=Property, numberlike=thmx, style=remark]{property}

\numberwithin{equation}{section}
\allowdisplaybreaks[1]
\usepackage[noabbrev]{cleveref} 
\crefname{figure}{Figure}{Figures}
\crefname{table}{Table}{Tables}
\crefname{thm}{Theorem}{Theorems}
\crefname{thmx}{Theorem}{Theorems}
\crefname{lem}{Lemma}{Lemmas}
\crefname{defi}{Definition}{Definitions}
\crefname{setup}{Setup}{Setups}
\crefname{conj}{Conjecture}{Conjectures}
\crefname{quest}{Question}{Questions}
\crefname{cor}{Corollary}{Corollaries}
\crefname{corx}{Corollary}{Corollaries}
\crefname{prop}{Proposition}{Propositions}
\crefname{ex}{Example}{Examples}
\crefname{rem}{Remark}{Remarks}
\crefname{section}{Section}{Sections}
\crefname{subsection}{Subsection}{Subsections}
\crefname{chapter}{Chapter}{Chapters}
\crefname{appendix}{Appendix}{Appendices}
\crefname{property}{Property}{Properties}
\crefdefaultlabelformat{#2\textup{#1}#3}
\creflabelformat{enumi}{(#2#1#3)}
\usepackage{xparse}
\usepackage{ourmacros}
\newcommand{\Sc}{\scal}
\newcommand{\wid}{\width}

\newcommand{\E}{\mathcal{E}}
\newcommand{\Frthl}{\mathcal{F}}
\newcommand{\p}{\partial}

 \newcommand{\eps}{\varepsilon}
\newcommand{\Ha}{\mathcal{H}}
\newcommand{\NPSC}{\mathrm{NPSC^+}}

\usepackage{authblk}

\title{Nonnegative scalar curvature on manifolds with at least two ends}
\author{Simone Cecchini\thanks{Funded by the Deutsche Forschungsgemeinschaft (DFG, German Research Foundation) through the Priority Programme ``Geometry at Infinity'' (SPP 2026, CE~393/1-1).}}
\affil{Texas A\&M University\\USA\\\href{mailto:cecchini@tamu.edu}{cecchini@tamu.edu}\\\href{https://simonececchini.org}{simonececchini.org}}
\affil{University of Göttingen\\Mathematisches Institut\\Germany}
\author{Daniel Räde\thanks{Funded by a Doctoral Scholarship of the 'Studienstiftung des deutschen Volkes'.}}
\affil{Augsburg University\\Institut f\"ur Mathematik\\Germany\\\href{mailto:draede.math@gmail.com}{draede.math@gmail.com}}
\author{Rudolf Zeidler\thanks{Funded by the Deutsche Forschungsgemeinschaft (DFG, German Research Foundation) – Project-ID 427320536 – SFB 1442, as well as under Germany’s Excellence Strategy EXC 2044  390685587, Mathematics Münster: Dynamics–Geometry–Structure, and through the Priority Programme ``Geometry at Infinity'' (SPP 2026, ZE~1123/2-2).}}
\affil{University of Münster\\Mathematisches Institut\\Germany\\\href{mailto:math@rzeidler.eu}{math@rzeidler.eu}\\\href{https://www.rzeidler.eu}{www.rzeidler.eu}}
\date{}
\hypersetup{
  colorlinks=true,
  allcolors=blue,
}
\begin{document}

\maketitle

\begin{abstract}
Let $M$ be an orientable connected $n$-dimensional manifold with $n\in\{6,7\}$ and let $Y\subset M$ be a two-sided closed connected incompressible hypersurface which does not admit a metric of positive scalar curvature (abbreviated by psc). 
Moreover, suppose that the universal covers of $M$ and $Y$ are either both spin or both non-spin.
Using Gromov's $\mu$-bubbles, we show that $M$ does not admit a complete metric of psc.
We provide an example showing that the spin/non-spin hypothesis cannot be dropped from the statement of this result.
This answers, up to dimension $7$, a question by Gromov for a large class of cases.
Furthermore, we prove a related result for submanifolds of codimension two.
We deduce as special cases that, if $Y$ does not admit a metric of psc and $\dim(Y) \neq 4$, then $M := Y\times\mathbb{R}$ does not carry a complete metric of psc and $N := Y \times \mathbb{R}^2$ does not carry a complete metric of uniformly psc provided that $\dim(M) \leq 7$ and $\dim(N) \leq 7$, respectively.
This solves, up to dimension $7$, a conjecture due to Rosenberg and Stolz in the case of orientable manifolds.
\end{abstract}



\section{Introduction}

In the 1994 survey article~\cite[Section 7]{RoS94}, Rosenberg and Stolz proposed the following conjectures concerning the (non-)existence of \emph{positive scalar curvature} (abbreviated by \emph{psc}) metrics on certain non-compact manifolds:

\begin{conj}\label{conj:main1}
Let $Y$ be a closed manifold of dimension $(n-1) \neq4$ which does not admit a metric of positive scalar curvature.
Then $Y\times \R$ does not admit a complete metric of positive scalar curvature.
\end{conj}

\begin{conj}\label{conj:main2}
Let $Y$ be a closed manifold of dimension $(n-2) \neq4$ which does not admit a metric of positive scalar curvature.
Then $Y\times \R^2$ does not admit a complete metric with uniformly positive scalar curvature.
\end{conj}

Recently Gromov phrased a related conjecture \cite[11.12, Conjecture C]{Gro18} which is a cornerstone of his program regarding the study of metric inequalities with scalar curvature:

\begin{conj}\label{conj:bands}
Let $Y$ be a closed manifold of dimension $(n-1)\neq4$ which does not admit a metric of positive scalar curvature and $X=Y\times[-1,1]$.
If $g$ is a Riemannian metric on $X$ with $\Sc(X,g)\geq n(n-1)$, then
\[\wid(X,g):=\dist_g(Y\times\{-1\},Y\times\{1\})\leq\frac{2\pi}{n}.\]
\end{conj}

\begin{rem}\label{rem:counterexamples}
The condition \(\dim(Y) \neq 4\) (which we have added in our formulation of \cref{conj:main1,conj:main2}) is necessary since in this case there are well known counterexamples using Seiberg--Witten obstructions to positive scalar curvature. 
It is possible to show that there exists a closed simply connected 4-manifold $Y$ which does not admit a metric of positive scalar curvature while $Y\times \Sphere^1$ does (see \cite[Counterexample 4.16]{RosenbergStolz:PSCSurgeryConnections}). But then $Y\times \R$ (and consequently $Y\times\R^2$) admits a complete metric with uniformly positive scalar curvature which would violate \cref{conj:main1,conj:main2,conj:bands}.
\end{rem}

As is common in the study of scalar curvature, there are two broad families of methods to approach these conjectures: One is based on the spinor Dirac operator on spin manifolds originating in the work of Lichnerowicz~\cite{Lichnerwociz:Spineurs}.
Indeed, using different variants of index theory on non-compact manifolds, it is now well established that \cref{conj:main1,conj:main2} can be proved whenever \(Y\) admits an index-theoretic obstruction to positive scalar curvature such as the Rosenberg index~\cite{Rosenberg:PSCNovikovI}, see~\cite{Cecchini:CalliasTypePSC,HankePapeSchick:CodimensionTwoIndex,Zeidler:IndexObstructionPositive}.
The other family of methods is based on geometric measure theory and originates in the minimal hypersurface method of \citeauthor{SchoenYau:HypersurfaceMethod}~\cite{SchoenYau:HypersurfaceMethod}.
The first established cases of \cref{conj:bands} by Gromov~\cite[Section~2]{Gromov:MetricInequalitiesScalar} used the classical minimal hypersurface method.
Subsequently also a Dirac operator approach to \cref{conj:bands} was developed by \citeauthor{Cecchini-Zeidler:ScalarMean}~\cite{Ce20,Zeidler:band-width-estimates,ZeidlerWidthLargeness,Cecchini-Zeidler:ScalarMean}.

In \cite[Sections 3.6, 5]{gromovFourLecturesScalar2019v6}, Gromov proposes a different approach towards \cref{conj:bands} using a modified version of the minimal hypersurface method involving so called $\mu$-bubbles.
Following this idea, \citeauthor{Rae21}~\cite{Rae21} proved \cref{conj:bands} and generalizations thereof in case that $Y$ is orientable and $n\leq 7$.
\Cref{conj:main1,conj:main2} have not been directly approached via minimal hypersurface techniques so far, in particular due to the non-compactness inherent to the problem.
However, in recent work of various authors (e.g.~\cite{lesourd2020positive,Lesourd-Unger-Yau:Positive_mass_ends,Zhu:RigdityComplete,chodosh_li}), \(\mu\)-bubbles have turned out to be a useful tool to deal with non-compact situations.
Even though there is no direct formal implication between \cref{conj:main1} and \cref{conj:bands}, it was observed in \cite{ZeidlerWidthLargeness} that the Dirac operator methods used in \cites{Ce20, Zeidler:band-width-estimates} to attack Conjecture \ref{conj:bands} can be refined to prove a more general statement (compare~\cite[Conjecture~7.1]{ZeidlerWidthLargeness}) which imply both \cref{conj:main1,conj:bands} for closed spin manifolds with non-vanishing Rosenberg index.

The main objective of this article is to combine the ideas from \cite{ZeidlerWidthLargeness} with \(\mu\)-bubble methods, in particular from \cite{Rae21}, to prove a generalization of \cref{conj:main1} in case $Y$ is orientable and of dimension $\leq 6$.
By a reduction argument to a codimension one situation (compare~\cites[Theorem~7.5]{GromovLawson:PSCDiracComplete}{HankePapeSchick:CodimensionTwoIndex}), we also establish \cref{conj:main2} in case $Y$ is orientable and of dimension $\leq5$.
We note that ideas related to \cref{conj:main1} have also appeared recently in the work of \citeauthor{CPSZ21}~\cite{CPSZ21} in connection with the positive mass theorem.

More generally than \cref{conj:main1}, one may ask under which circumstances the existence of a hypersurface \(Y \subset M\) which does not admit psc already is an obstruction to the existence of a complete psc metric on the ambient manifold \(M\).
This question has been discussed by Gromov in \cite[\S11.6]{Gromov:MetricInequalitiesScalar}, where in particular he asked if it would be enough to assume that \(Y\) is a two-sided \emph{incompressible hypersurface}, that is, the map \(\pi_1 Y \to \pi_1 M\) induced by the inclusion is injective.
In the case that the ambient manifold is spin of dimension \(n \in \{6,7\}\) and under further geometric conditions, a proof confirming this was sketched in \cite[708\psqq]{Gromov:MetricInequalitiesScalar} and it was asked if the spin hypothesis can be dropped.
We answer this question in our first theorem below together with the following \cref{ex:counterexample}.
Here we say that a connected manifold \(M\) is \emph{almost spin}, if its universal covering \(\tilde{M}\) is spin, and we say it is \emph{totally non-spin} if \(\tilde{M}\) is non-spin.
Since spin structures lift to coverings, being almost spin is equivalent to the existence of some  covering which is spin.
Unless explicitly stated otherwise, we consider manifolds without boundary.
\begin{thm}\label{thm:main_theorem}
  Let \(M\) be an orientable connected \(n\)-dimensional manifold with \(n\in\{6,7\}\) and let \(Y \subset M\) be a two-sided closed connected incompressible hypersurface which does not admit a metric of positive scalar curvature.
  Suppose that one of the following two conditions holds:
  \begin{myenuma}
    \item \(M\) is almost spin. \label{item:spin}
    \item \(Y\) is totally non-spin. \label{item:non_spin}
  \end{myenuma}
  Then \(M\) does not admit a complete metric of positive scalar curvature.
  More precisely, if \(g\) is a complete metric of non-negative scalar curvature on \(M\), then \((M,g)\) admits a connected Riemannian covering isometric to \((N \times \R, g_N + \D{x}^2)\), where \((N,g_N)\) is a closed Ricci flat manifold.
\end{thm}
Note that if \(M\) is almost spin, then any two-sided hypersurface \(Y \subset M\) is almost spin itself.
Conversely, if a two-sided hypersurface is totally non-spin, then so is the ambient manifold.
Thus the alternative hypotheses \labelcref{item:spin,item:non_spin} of \cref{thm:main_theorem} simply say that either \(M\) and \(Y\) are both almost spin or neither is.
The following example shows that this restriction cannot be dropped.

\begin{ex}\label{ex:counterexample}
  Fix \(n \geq 6\).
  Let \(L\) be the K3 surface, that is, a \(4\)-dimensional simply-connected spin manifold such that \(\Ahat(L) \neq 0\).
  Consider the closed manifold \(M \coloneqq (L \times \Torus^{n-4}) \# (\CP^2 \times \Sphere^{n-4})\).
  Then, since \(L\) is oriented cobordant to a psc manifold by \cite[\S 3]{GromovLawson:Classification}, the cobordism class represented by \(M\) in \(\Omega^{\SO}_{n}(\Bfree \Z^{n-4})\) has a psc representative.
  Thus the totally non-spin manifold \(M\) itself admits a psc metric, e.g.\ by~\cite[Theorem~2.13]{Rosenberg:PSCNovikovII} or \cite[Theorem~1.5]{Ebert-Frenck:GLC_surgery}.\footnote{Alternatively, the existence of a psc metric on \(M\) can also be deduced from \cite[Theorem~5.6]{Miyazaki:PSCExistence}.}
  It contains \(Y = L \times \Torus^{n-5}\) as an incompressible hypersurface which does not admit psc by~\cite[Corollary~5.22]{GromovLawson:PSCDiracComplete}.
  By passing to the covering \(\hat{M} \to M\) with \(\pi_1 \hat{M} = \Z^{n-5}\) corresponding to \(\pi_1 Y\), we even find an example of a complete manifold which admits a complete uniform psc metric although it contains a closed incompressible \emph{separating} hypersurface which does not.
\end{ex}
This shows that the almost spin condition is fundamentally relevant for this kind of problem even though the proof \cref{thm:main_theorem} does not use the Dirac operator at all.
Moreover, even in the spin case, \cref{thm:main_theorem} is stronger (in the dimension range where it applies) than what can be proved using index-theoretic techniques because there are examples of closed spin manifolds which do not admit psc even though all known index invariants vanish~\cite{Schick:Counterexample}.

The upper dimension bound \(n \leq 7\) comes from the usual problem pertaining to singularities of minimal hypersurfaces and \(\mu\)-bubbles. 
Thus, if this issue was resolved, the upper dimensional bound in \cref{thm:main_theorem} (and in all other results of this paper) should conjecturally be removable.
The case \(n = 8\) likely is more easily accessible via an adaptation of the work of N.~Smale~\cite{Smale1993Regular8Dim}.

\begin{rem}
In~\cite[\S11.6]{Gromov:MetricInequalitiesScalar}, Gromov conjectured that in general, given a complete Riemannian manifold $M$, the existence of a two-sided incompressible closed embedded hypersurface which does not carry a psc metric obstructs the existence of a function $h$ on $M$ such that
\[
    \frac{n}{n-1}h^2-2\left|\D h\right|+\scal_g\geq0.
\]
This condition is motivated by the potential function used to construct $\mu$-bubbles.
From this point of view, ~\cref{thm:main_theorem} gives a complete description of the case when $h=0$, which includes many geometrically interesting cases.
We point out that our method can be adapted to treat cases with non-trivial $h$.
In the direction of Gromov's motivating example~\cite[708\psqq]{Gromov:MetricInequalitiesScalar}, compare \cref{thm:analysis1}.
We also point out that, compared to the method sketched in~\cite[\S11.6]{Gromov:MetricInequalitiesScalar}, our technique does not require any extra assumption on the geometry of the manifold.
\end{rem}

For \(n = 5\), \cref{thm:main_theorem} fails as mentioned in \cref{rem:counterexamples}.
On the other hand, for \(2 \leq n \leq 4\), it holds even independently of the hypotheses \labelcref{item:spin,item:non_spin}, but for a different reason than in high dimensions.
To formulate this, we consider the following notion:
\begin{defi}[{compare~\cites[Definition~2.20]{raede2021scalar}[\enquote{\(\mathcal{C}_{\mathrm{deg}}\)}]{CPSZ21}}]\label{defi:NPSC}
  A closed connected oriented manifold \(Y\)  is called \(\NPSC\) if it satisfies the following property: No closed oriented manifold \(Z\) which admits a continuous map of non-zero degree \(Z \to Y\) admits a metric of positive scalar curvature.
\end{defi}
For instance, it has recently been shown~\cites{chodosh_li,chodosh2021classifying, Gromov:5d} that closed oriented aspherical manifolds of dimension $\leq5$ are \(\NPSC\).
Moreover, a closed oriented manifold of dimension \(\leq 3\) which does not admit psc necessarily admits a non-zero degree map to an aspherical \(\NPSC\) manifold.\footnote{In dimension \(2\) this is a consequence of the classification of surfaces and the Gauß--Bonnet theorem, whereas in dimension \(3\) it is a consequence of the classification of \(3\)-manifolds which admit psc following from Perelman's work, see e.g.\ the discussion in~\cite{Marques2012:DeformingThreeManifolds}.}
Thus the low-dimensional counterpart to \cref{thm:main_theorem} is contained in the following theorem which already appeared recently in \cite{CPSZ21}.
\begin{thm}[{compare~\cite[Theorem~1.1]{CPSZ21}}]\label{thm:npsc_theorem}
  Let \(M\) be an orientable connected \(n\)-dimensional manifold with \(n \leq 7\) and let \(\iota \colon Y \hookrightarrow M\) be a two-sided closed hypersurface that admits a map of non-zero degree \(\phi \colon Y \to Y_0\) to an aspherical \(\NPSC\) manifold \(Y_0\) and such that \(\ker(\pi_1 Y \xrightarrow{\iota_\ast} \pi_1 M) \subseteq \ker(\pi_1 Y \xrightarrow{\phi_\ast} \pi_1 Y_0)\).
  Then \(M\) does not admit a complete metric of positive scalar curvature.
  More precisely, if \(g\) is a complete metric of non-negative scalar curvature on \(M\), then \((M,g)\) admits a connected Riemannian covering isometric to \((N \times \R, g_N + \D{x}^2)\), where \((N,g_N)\) is a closed Ricci flat manifold.
\end{thm}

In particular, applying \cref{thm:main_theorem,thm:npsc_theorem} to \(M = Y \times \R\), we deduce: 

\begin{cor} \label{cor:rosenberg-stolz-codim1}
  \cref{conj:main1} holds for orientable manifolds in dimensions \( 5 \neq n \leq 7\).
\end{cor}
However, we note that the low-dimensional cases \(n \leq 4\) of \cref{cor:rosenberg-stolz-codim1} were already known because an oriented manifold of dimension \(\leq 3\) which does not admit positive scalar curvature is necessarily spin and has rationally non-vanishing Rosenberg index, and so the situation is within the scope of index-theoretic results such as \cite[Theorem~A]{Cecchini:CalliasTypePSC}.

We now turn to our results corresponding to \cref{conj:main2}:

\begin{thm}\label{thm:main_theorem_codim2}
  Let \(M\) be an orientable connected \(7\)-dimensional manifold and let \(Y \subset M\) be a closed connected \(5\)-dimensional submanifold with trivial normal bundle such that the inclusion induces an injection \(\pi_1 Y \hookrightarrow \pi_1 M\) and a surjection \(\pi_2 Y \twoheadrightarrow \pi_2 M\).
  Suppose that \(Y\) does not admit a metric of positive scalar curvature.
  Then \(M\) does not admit a complete metric of uniformly positive scalar curvature.
\end{thm}

Note that, unlike \cref{thm:main_theorem}, we do not need to impose any conditions involving spin structures.
The intuitive reason for this is that the hypotheses already imply that the induced map \(\tilde{Y} \to \tilde{M}\) between universal covers is \(2\)-connected and so \(M\) is almost spin if and only if \(Y\) is.
On the other hand, the surjectivity condition \(\pi_2 Y \twoheadrightarrow \pi_2 M\) cannot be omitted as the example \(M = Y \times \Sphere^2\) shows.

We also have a codimension \(2\) counterpart to \cref{thm:npsc_theorem} exploiting the \(\NPSC\) property.

\begin{thm}\label{thm:npsc_theorem_codim2}
  Let \(M\) be an orientable connected \(n\)-dimensional manifold with \(n \leq 7\) and let \(Y \subset M\) be a closed connected \((n-2)\)-dimensional submanifold with trivial normal bundle such that the inclusion induces an injection \(\pi_1 Y \hookrightarrow \pi_1 M\) and a surjection \(\pi_2 Y \twoheadrightarrow \pi_2 M\).
  Suppose that \(Y\) admits a map of non-zero degree \(Y \to Y_0\) to an aspherical \(\NPSC\) manifold \(Y_0\).
  Then \(M\) does not admit a complete metric of uniformly positive scalar curvature.
\end{thm}

Similarly as before, this includes a version of \cref{thm:main_theorem_codim2} in dimensions \(n \leq 5\):

\begin{cor}
  Let \( 6 \neq n \leq 7\).
  Let \(M\) be an orientable connected \(n\)-dimensional manifold and let \(Y \subset M\) be a closed connected \((n-2)\)-dimensional submanifold with trivial normal bundle such that the inclusion induces an injection \(\pi_1 Y \hookrightarrow \pi_1 M\) and a surjection \(\pi_2 Y \twoheadrightarrow \pi_2 M\).
  Suppose that \(Y\) does not admit a metric of positive scalar curvature.
  Then \(M\) does not admit a complete metric of uniformly positive scalar curvature.
\end{cor}
\begin{proof}
  For \(n = 7\), this is a restatement of \cref{thm:main_theorem_codim2}.
  If \(n \leq 5\), then \(\dim(Y) \leq 3\) and so it admits a non-zero degree map \(Y \to Y_0\), where \(Y_0\) is aspherical and \(\NPSC\).
  Thus \cref{thm:npsc_theorem_codim2} is applicable in this case.
\end{proof}

Specializing to \(M = Y \times \R^2\), we finally obtain:

\begin{cor}\label{cor:rosenberg-stolz-codim2}
  \cref{conj:main2} holds for orientable manifolds in dimensions \( 6 \neq n \leq 7\).
\end{cor}

Again we like to point out that the low-dimensional cases \(n \leq 5\) of \cref{cor:rosenberg-stolz-codim2} were already known due to index-theoretic results~\cites{HankePapeSchick:CodimensionTwoIndex}[Theorem~1.10]{Zeidler:band-width-estimates}.

This article is organized as follows: In \cref{sec:Bands}, we prepare an abstract setup for the study of manifolds with at least two ends which underpins our work.
In \cref{sec:part_comparison}, we state quantitative comparison results in the spirit of \cref{conj:bands} which are then used together with topological arguments in \cref{sec:obstr_bands} to prove our codimension one results.
The codimension two results are deduced in \cref{sec:codimension_two}.
Finally, in \cref{sec:proof_part_comparison} we provide the analytic proofs of the comparison statements from \cref{sec:part_comparison}.
\subsection*{Acknowledgements}
The authors would like to thank Georg Frenck for helpful discussions and suggestions as well as the anonymous referee for useful comments.

\section{Bands}\label{sec:Bands}
A natural class of manifolds generalizing the situation of \cref{conj:main1} are connected non-compact manifolds, where we partition the set of ends into two parts and we consider hypersurfaces separating these parts from each other.
To make this precise, we will make use of the Freudenthal end compactification~\cite{Freudenthal:Enden} of a connected manifold \(M\), denoted by \(\Frthl M = M \cup \E M\), where \(\E M\) is the space of ends.
By definition, the space of ends is the inverse limit \(\E M = \varprojlim \pi_0(M \setminus K)\), where \(K\) runs over compact subsets of \(M\) and each set of components \(\pi_0(M \setminus K)\) is endowed with the discrete topology.
\begin{defi}\label{def:openband}\ 
\begin{myenumi}
\item An \emph{open band} is a connected non-compact manifold $M$ without boundary and a decomposition
\[\E M=\E_-M\sqcup\E_+M,\]
where $\E_\pm M$ are non-empty closed\footnote{While the space of ends is totally disconnected, it is in general not necessarily discrete.
In this case it is important to assume \(\E_\pm M\) to be closed (and thus clopen) subsets of \(\E M\).} subsets \(\E_\pm M \subset \E M\); in particular $M$ has at least two ends.
\item Given an open band \(M\), a \emph{separating hypersurface} \(\Sigma \subset M\) is a compact hypersurface which separates each end in \(\E_-M\) from every end in \(\E_+ M\). Moreover, we say a separating hypersurface \(\Sigma \subset M\) is \emph{properly separating} if every component of \(\Sigma\) can be connected to both \(\E_+ M\) and \(\E_- M\) inside \(M \setminus \Sigma\).
\item Let \(\Sigma_-, \Sigma_+ \subset M\) be two properly separating hypersurfaces in an open band \(M\). 
Then we write \(\Sigma_- \prec \Sigma_+\) if the hypersurface \(\Sigma_-\) is contained in the union of those components of \(M \setminus \Sigma_+\) that contain the ends in \(\E_- M\) (or equivalently, \(\Sigma_+\) is contained in the union of those components of \(M \setminus \Sigma_-\) containing \(\E_+ M\)).
\end{myenumi}
\end{defi}

The condition of being \emph{properly} separating simply means that there are no superfluous components as the observation recorded in the following lemma illustrates.

\begin{lem}\label{lem:proper_separating}
  Let \(\Sigma \subset M\) be a separating hypersurface in an open band \(M\).
  Then there exists a union of components of \(\Sigma\) which is a properly separating hypersurface in \(M\).
\end{lem}
\begin{proof}
  Suppose that \(\Sigma\) is a separating hypersurface that contains a component not connected to both \(\E_- M\) and \(\E_+ M\) inside \(M \setminus \Sigma\). 
  Then the hypersurface \(\Sigma'\) obtained from \(\Sigma\) by deleting this component is still a separating hypersurface.
  This shows that a minimal collection of components of \(\Sigma\) such that its union is still separating yields the desired properly separating hypersurface.
\end{proof}

The next elementary lemma shows that (properly) separating hypersurfaces always exist and we can find them arbitrarily far out.

\begin{lem}\label{lem:exist_separating}
Let \(M\) be an open band and \(K \subset M\) be an arbitrary compact subset.
Then there exists a properly separating hypersurface \(\Sigma \subset M\) which also separates \(K\) from \(\E_+ M\) (or \(\E_- M\), respectively).
\end{lem}
\begin{proof}
Note that the end compactification \(\Frthl M\) is a compact Hausdorff space which in our case of a connected manifold is also second countable.
Thus, since \(\E_\pm M\) are two disjoint closed subsets of \(\Frthl M\), Urysohn's lemma implies the existence of a continuous function \(f \colon \Frthl M \to [-1,1]\) such that \(\E_\pm M = f^{-1}(\pm 1)\).
Since \(K \subseteq M\) is compact, there exists \(0 < r < 1\) such that \(K \subseteq f^{-1}([-r,r])\).
Choose \(s \in (r,1)\).
Then \(f^{-1}(s) \subseteq M\) is a compact subset which separates \(\E_- M\) from \(\E_+ M\).
Now choose a connected compact \(n\)-dimensional submanifold \(V \subset M\) with boundary, where \(n = \dim(M)\), such that \(f^{-1}(s) \subseteq \mathring{V} \subseteq f^{-1}([s - \varepsilon, s+ \varepsilon])\) for some \(\varepsilon > 0\) with \(r < s - \varepsilon\).
Then \(\partial V\) is a separating hypersurface and it contains a properly separating hypersurface \(\Sigma \subseteq \partial V\) by \cref{lem:proper_separating}.
Since by construction \(f(x) \leq r < s - \varepsilon \leq f(y) \leq s + \varepsilon < 1\) for each \(x \in K\) and \(y \in \Sigma\), it follows that \(\Sigma\) must separate \(K\) from \(\E_+ M\).
A completely analogous argument also provides a properly separating hypersurface that separates \(K\) from \(\E_- M\).
\end{proof}

In the spirit of \cref{conj:main1} we will be interested in open bands with:

\begin{property}\label{pB}
No separating hypersurface admits a metric of positive scalar curvature.
\end{property}

We will also work with compact bands, which may be viewed as a special case of open bands in the following sense:

\begin{defi}\label{def:band}
A \emph{compact band}, to which we will often simply refer to as a \enquote{\emph{band}}, is a connected compact manifold $X$ together with the structure of an open band on its interior \(\mathring{X}\). 
In other words, this amounts to a decomposition \(\p X=\p_-X\sqcup\p_+X\), where $\p_\pm X$ are (non-empty) unions of boundary components.
\end{defi}

The notions of (properly) separating hypersurfaces and \cref{pB} make sense for compact bands by considering them for the interior.

We also note that, if $M$ is an open band and $\Sigma_\pm$ are two properly separating hypersurfaces such that \(\Sigma_- \prec \Sigma_+\), then the hypersurfaces $\Sigma_\pm$ bound a compact band \(X \subset M\).
In this case, if \(M\) has \cref{pB}, then so has \(X\).

\section{The partitioned comparison principle}
\label{sec:part_comparison}
In this section, we establish the main analytic results on which our main theorems rely.
The central concept we study here is the $\wid(X,g)$ of a compact band \(X\) endowed with a Riemannian metric \(g\), that is, the distance between $\p_-X$ and $\p_+X$ with respect to \(g\).
As is explained in \cite[Section 3.6]{gromovFourLecturesScalar2019v6} the width of an $n$-dimensional Riemannian band $(X,g)$ with $\Sc\geq n(n-1)$ is bounded from above by $\frac{2\pi}{n}$ if $X$ has \cref{pB} and $n\leq7$.
We will work in the setting where $(X,g)$ has \cref{pB} and is partitioned into multiple segments with possibly different lower scalar curvature bounds. It turns out that in this case positivity of the scalar curvature in a single segment can often have global effects on the geometry of $(X,g)$.

\begin{defi}
Let $X$ be a compact band and let $\Sigma_i$, for $i\in\{1,\dots, k\}$, be properly separating hypersurfaces such that \(\Sigma_{i-1} \prec \Sigma_i\) for $1 < i \leq k$. We call the triple $(X,\Sigma_i,k)$ a \emph{partitioned band} and denote by $V_j$, for $j\in\{1,\ldots, k+1\}$, the segment of $X$ bounded by $\Sigma_{j-1}$ and \(\Sigma_j\), where \(\Sigma_0 = \partial_- X\) and \(\Sigma_{k+1} = \partial_+ X\).
\end{defi}

\begin{defi}\label{def:logconcave}
A smooth function $\varphi \colon [a,b]\to\R_+$ is called $\log$-\emph{concave} if 
\[\frac{\D{}^2}{\D{t}^2}\log(\varphi(t))=\left(\frac{\varphi'(t)}{\varphi(t)}\right)'\leq0\]
for all $t\in[a,b]$.
If the inequality is strict, then $\varphi$ is called \emph{strictly $\log$-concave}.
\end{defi}

\begin{defi}\label{def:model}
Let $(N,g_N)$ be a a closed scalar flat Riemannian manifold. A warped product
\[(M,g_\varphi)=(N\times[a,b],\varphi^2(t)g_N+\D{t}^2)\]
is called a \emph{model space} if $\Sc(M,g_\varphi)$ is constant and $\varphi$ is strictly $\log$-concave.
\end{defi}

\begin{rem}
This notion of a model space arises in the scalar and mean curvature comparison theory surrounding Conjecture \ref{conj:bands}.
In \cites{Cecchini-Zeidler:ScalarMean,Rae21} a compact Riemannian band $(X,g)$ is compared, in scalar curvature, mean curvature and width, to a warped product $(M,g_\varphi)$ over an arbitrary scalar flat manifold $(N,g_N)$ with strictly $\log$-concave warping function.
It turns out that, if $X$ has Property \ref{pB}, $\Sc(X,g)\geq\Sc(M,g_\varphi)$ and $\mean(\p X,g)\geq \mean(\p M,g_\varphi)$, then $\width(X,g)\leq\width(M,g_\varphi)$.
Furthermore, Cecchini and Zeidler \cite[Theorem 8.3, Theorem 9.1]{Cecchini-Zeidler:ScalarMean} showed that, under further topological assumptions on $X$, equality of widths can only be achieved if $(X,g)$ itself is isometric to such a warped product.
In this section we adapt these ideas and compare a compact partitioned Riemannian band with Property \ref{pB} to a sequence of model spaces, whose scalar curvatures may differ, but whose mean curvatures fit together pairwise. 
\end{rem}

Let $(N,h)$ be a compact Riemannian manifold with boundary and let $\p_jN$ be a connected component of the boundary of $N$.
Then we denote by $\mean(\p_j N,h)$ the mean curvature of $\p_jN$ with respect to the interior unit normal field.
According to our convention, the boundary of the unit disk in $\R^3$ has mean curvature $2$.
In the following theorem, we give estimates on the widths of segments inside a partitioned band based on scalar and mean curvature. 

\begin{thm}\label{thm:general}
Let $n\leq7$ and let $(X,\Sigma_i,k)$ be an orientable partitioned $n$-dimensional band with \cref{pB}. 
For $j\in\{1,\ldots,k+1\}$, let $(M_j,g_{\varphi_j})$ be model spaces over a fixed closed scalar flat Riemannian manifold $(N,g)$. If
\begin{itemize}
\item[$\triangleright$] $\Sc(V_j,g)\geq\Sc(M_j,g_{\varphi_j})$ for all \(j\in\{1,\ldots,k+1\}\),
\item[$\triangleright$] $\mean(\p_-X,g)\geq \mean(\p_-M_1,g_{\varphi_1})$ and $\mean (\p_+X,g)\geq \mean(\p_+M_{k+1},g_{\varphi_{k+1}})$,
\item[$\triangleright$] $\mean(\p_+M_j,g_{\varphi_j})=-\mean(\p_-M_{j+1},g_{\varphi_{j+1}})$ for all $j\in\{1,\ldots,k\}$, 
\end{itemize}
then $\wid(V_j,g)\leq\wid(M_j,g_{\varphi_j})$ for at least one $j\in\{1,\ldots,k+1\}$.
\end{thm}

The following result is more or less a direct application of Theorem \ref{thm:general}:

\begin{cor}\label{thm:analysis}
Let $n\leq7$ and let $(X,\Sigma_i,2)$ be an orientable partitioned $n$-dimensional band with \cref{pB}.
Let $g$ be a metric on $X$ and let $\kappa$ be a positive constant. If 
\begin{itemize}
\item[$\triangleright$] $\Sc(V_2,g)\geq\kappa n(n-1)$,
\item[$\triangleright$] $\Sc(X,g)\geq0$,
\end{itemize}
then $\min\{\wid(V_1,g),\wid(V_3,g)\}< \ell=\frac{2}{\sqrt{\kappa}n}\cot\left(\frac{\sqrt{\kappa}nd}{4}\right)$, where $d \coloneqq \wid(V_2,g)<\frac{2\pi}{\sqrt{\kappa}n}$.
\end{cor}

If, instead of $\Sc(X,g)\geq0$, one assumes that the scalar curvature of the partitioned band is bounded from below by a negative constant, \cref{thm:general} provides the following estimate, which is very much in the same spirit of~\cref{thm:analysis} and should be compared with Zeidler's result \cite[Theorem 1.4]{ZeidlerWidthLargeness} in the spin setting.

\begin{cor}\label{thm:analysis1}
Let $n\leq7$ and let $(X,\Sigma_i,2)$ be an orientable partitioned $n$-dimensional band with \cref{pB}.
Let $g$ be a metric on $X$ and let $\kappa$ be a positive constant. If 
\begin{itemize}
\item[$\triangleright$] $\Sc(V_2,g)\geq\kappa n(n-1)$,
\item[$\triangleright$] $\Sc(X,g)\geq-\sigma>-\kappa n(n-1)\tan\left(\frac{\sqrt{\kappa}nd}{4}\right)^2$, where $d \coloneqq \wid(V_2,g)<\frac{2\pi}{\sqrt{\kappa}n}$,
\end{itemize}
then $\min\{\wid(V_1,g),\wid(V_3,g)\}< \ell$, where $\ell$ is such that
\[\sqrt{\kappa}(n-1)\tan\left(\frac{\sqrt{\kappa}nd}{4}\right)=\sqrt{\frac{\sigma(n-1)}{n}}\coth\left(\frac{\sqrt{\sigma n}\ell}{2\sqrt{n-1}}\right).\]
\end{cor}

We postpone the proofs of \cref{thm:general,thm:analysis,thm:analysis1} to \cref{sec:proof_part_comparison}.

\section{Obstructions on open bands}\label{sec:obstr_bands}

We want to use \cref{thm:analysis} or \cref{thm:analysis1} to attack Conjecture \ref{conj:main1}.
If $X=Y\times\R$ and $g$ is a complete metric on $X$, we consider the compact segment $Y\times[-C,C]$ for any $C>1$, which is partitioned into the bands $Y\times[-C,-1]$, $Y\times[-1,1]$ and $Y\times[1,C]$.
If the scalar curvature of $(X,g)$ is assumed to be positive and $Y\times[-C,C]$ has Property \ref{pB}, the minimum of the widths of $(Y\times[-C,-1],g)$ and $(Y\times[1,C],g)$ is bounded from above in terms of $\wid(Y\times[-1,1],g)$ and the infimum of $\Sc(Y\times[-1,1],g)$. 
For $C>1$ large enough, this produces a contradiction.

In this section, we formulate the most general result that one can prove in this manner based on our notion of an open band from \cref{def:openband}.
Indeed, \cref{thm:analysis} or \cref{thm:analysis1} will yield obstructions to the existence of a complete metric of positive scalar curvature on an open band with \cref{pB}. 
Moreover, since any open band $M$ has at least two ends by definition, we can also establish a rigidity result using the Cheeger--Gromoll splitting theorem.


\begin{thm}\label{thm:analysis2}
Let $n\leq7$ and let $M$ be an open $n$-dimensional band with \cref{pB}. If $g$ is a complete metric on $M$ with nonnegative scalar curvature, then $(M,g)$ is isometric to \((Y\times\R,g_Y+dt^2)\),
where $(Y,g_Y)$ is a closed Ricci flat manifold.
\end{thm}
\begin{proof}
If $(M,g)$ is not Ricci flat, then $M$ admits a complete metric $\hat{g}$ of positive scalar curvature by \cite[{Theorem B}]{Ka82}.
Let $\Sigma\subset M$ be a properly separating hypersurface which exists by \cref{lem:exist_separating}.
There is a $\kappa>0$ such that $\Sc(M,\hat{g})\geq\kappa n(n-1)$ in a neighborhood of \(\Sigma\) with width $d<\frac{2\pi}{\sqrt{\kappa}n}$, which is bounded by two properly separating hypersurfaces $\Sigma_1$ and $\Sigma_2$ such that \(\Sigma_1 \prec \Sigma \prec \Sigma_2\).
For every $C>0$, we can use \cref{lem:exist_separating} to find further properly separating hypersurfaces $\Sigma_\pm^C$ such that  \(\Sigma_-^C \prec \Sigma_1 \prec \Sigma_2 \prec \Sigma_+^C\) and \(\dist_{\hat{g}}(\Sigma_-^C,\Sigma_1) \geq C \leq \dist_{\hat{g}}(\Sigma_2,\Sigma_+^C)\).
Let $X^C$ be the compact band bounded by $\Sigma_\pm^C$.
Then for $C$ large enough, \cref{thm:analysis} or \cref{thm:analysis1} applied to $X^C$ yields a contradiction.
Hence $(X,g)$ is Ricci flat.

Since \(M\) is an open band it is disconnected at infinity and so admits a geodesic line (see e.g.~\cite[Lemma~7.3.1]{Petersen}).
Thus by the Cheeger--Gromoll splitting theorem (see e.g.~\cite[Theorem~7.3.5]{Petersen}), $(M,g)$ is isometric to $(Y\times\R,g_Y+dt^2)$, where $(Y,g_Y)$ is a Ricci flat manifold.
Furthermore, \(M\) has more than one end by the definition of an open band, and so \(Y\) must be compact. 
\end{proof}

Using a similar (and somewhat simpler argument), we also obtain the following statement:
\begin{prop}\label{prop:no_upsc_outside_compact}
  Let \(n \leq 7\) and let \(M\) be an open $n$-dimensional band with \cref{pB}. Then \(M\) does not admit a complete metric which has uniformly positive scalar curvature outside a compact subset.
\end{prop}
\begin{proof}
  Assume, by contradiction, that \(M\) admits a complete metric \(g\) which has scalar curvature \(\scal_g \geq \kappa n(n-1) > 0\) on \(M \setminus K\) for some compact subset \(K \subset M\) and some \(\kappa > 0\).
  Then, applying \cref{lem:exist_separating} twice, we can find a compact band \(X \subset M \setminus K\) which is bounded by two properly separating hypersurfaces \(\partial_- X \prec \partial_+ X \subset M\) such that \(\width(X,g) > \frac{2\pi}{\sqrt{\kappa}n}\).
  By construction, we also have \(\scal_g \geq \kappa n(n-1) > 0\) on \(X\).
  But this contradicts \cref{thm:analysis1} or even the usual band width estimate on compact bands \cites[Section 3.6]{gromovFourLecturesScalar2019v6}{raede2021scalar} of dimension \(\leq 7\).
\end{proof}

In the rest of this section, we investigate topological conditions which imply the existence of open bands with \cref{pB} and use them together with \cref{thm:analysis2} to prove the main \cref{thm:main_theorem,thm:npsc_theorem}.
Let us start with the notion of an incompressible hypersurface.

\begin{defi}\label{def:incompressible}
Let $X$ be a connected manifold and let $\Sigma\subset X$ be a connected hypersurface. 
We say that $\Sigma$ is incompressible if the induced map $\iota_*:\pi_1(\Sigma)\to \pi_1(X)$ is injective.
\end{defi}

We now provide topological conditions such that if an incompressible hypersurface does not admit psc, then the manifold is covered by an open band with \cref{pB}: see \cref{prop:tangential_structure}.
The arguments will be based on positive scalar curvature bordism and surgery techniques which usually require a case distinction depending on the presence of a spin structure.
To treat this in a unified way, we use the language of tangential structures, for details see e.g.~\cite[\S 5]{GTMW:CobordismCategory}.
We briefly recall that an \(n\)-dimensional tangential structure is given by a fibration \(\theta \colon B \to \Bfree\Orth(n)\) and a \(\theta\)-structure on an \(n\)-manifold \(M\) is a lift \(M \to B\) of the classifying map of the tangent bundle \(M \to \Bfree\Orth(n)\) along \(\theta \colon B \to \Bfree\Orth(n)\).
Then, given an \(n\)-dimensional tangential structure \(\theta \colon B \to \Bfree\Orth(n)\), one may define a notion of \(\theta\)-cobordism for \((n-1)\)-dimensional \(\theta\)-manifolds\footnote{To be precise, here one implicitly considers the \((n-1)\)-dimensional tangential structure obtained from \(\theta\) via the pullback along \(\BO(n-1) \to \BO(n)\).} and form the corresponding cobordism group which we denote by \(\Omega^{\theta}_{n-1}\).
Now the positive scalar curvature surgery principle can be phrased as follows:  Let \(\theta \colon B \to \BO(n)\) be an \(n\)-dimensional tangential structure and fix an \((n-1)\)-dimensional \(\theta\)-manifold \(N\) such that the map \(N \to B\) is \(2\)-connected. If \(N\) is \(\theta\)-cobordant to a \(\theta\)-manifold which admits a metric of psc, then \(N\) itself already admits psc, see~\cite[Theorem~1.5]{Ebert-Frenck:GLC_surgery}.

We now connect this back to the study of open bands and first observe that an open band endowed with a \(\theta\)-structure induces a well-defined \(\theta\)-cobordism class of \((n-1)\)-manifolds.
\begin{lem}\label{lem:properly_separating_cobordism_class}
  Let \(\theta \colon B \to \Bfree\Orth(n)\) be a tangential structure and let \(M\) be an open band endowed with a \(\theta\)-structure.
  Then any two properly separating hypersurfaces \(\Sigma_1, \Sigma_2 \subset M\) are \(\theta\)-cobordant.
\end{lem}
\begin{proof}
First consider the case that \(\Sigma_1 \prec \Sigma_2\).
Then the band in \(M\) bounded in between \(\Sigma_1\) and \(\Sigma_2\) is a cobordism witnessing that \(\Sigma_1\) and \(\Sigma_2\) are \(\theta\)-cobordant.
In general, \cref{lem:exist_separating} implies the existence of a properly separating hypersurface \(\Sigma' \subset M\) such that \(\Sigma_1 \prec \Sigma'\) and \(\Sigma_2 \prec \Sigma'\).
Thus, by the first case, both \(\Sigma_1\) and \(\Sigma_2\) are \(\theta\)-cobordant to \(\Sigma'\) which proves the desired statement.
\end{proof}

Next we use the psc surgery principle as stated above to derive a sufficient condition for the existence of open bands with \cref{pB}.

\begin{lem}\label{lem:tangential_structure}
  Let \(n \geq 6\) and let \(\theta \colon B \to \Bfree\Orth(n)\) be a tangential structure.
  Let \(M\) be a connected \(\theta\)-manifold without boundary and let \(Y \subset M\) be a closed connected two-sided hypersurface.
  Assume that \(Y\) does not support a psc metric and that the composition \(Y \subset M \to B\) induces an injection \(\pi_1 Y \hookrightarrow \pi_1 B\) and a surjection \(\pi_2 Y \twoheadrightarrow \pi_2 B\).
  Then the connected covering \(\hat{M}\) of \(M\) with \(\pi_1 \hat{M} = \pi_1 Y\) is an open band with \cref{pB}. 
\end{lem}
\begin{proof}
  First let \(\hat{B} \to B\) be the covering with \(\pi_1 \hat{B} = \pi_1 Y\).
  Then \(\hat{M} \to M\) is the pullback of the covering \(\hat{B} \to B\) along \(\theta \colon M \to B\).
  The embedding \(Y \subset M\) lifts to an embedding \(Y \subset \hat{M}\) as a hypersurface which separates \(\hat{M}\) into two components (otherwise there existed a loop in \(\hat{M}\) intersecting \(Y\) transversally in precisely one point which would contradict \(\pi_1 Y \cong \pi_1 \hat{M}\)).
  Moreover, we obtain a \(\hat{\theta}\)-structure on \(\hat{M}\), where \(\hat{\theta} \colon \hat{B} \to B \xrightarrow{\theta} \Bfree\Orth(n)\).
  \[
    \begin{tikzcd}
    & \hat{M} \ar[d] \ar[r]& \hat{B} \ar[d] \ar[dr, "\hat{\theta}"]\\
    Y \ar[r,hook] \ar[ur,hook]&M \ar[r] & B \ar[r,"\theta"] & \Bfree\Orth(n)
    \end{tikzcd}
    \]
    By assumption, the composition \(Y \hookrightarrow \hat{M} \to \hat{B}\) is \(2\)-connected.
    Let \(U_\pm\) be the two components of \(\hat{M} \setminus Y\).
    We observe that \(U_\pm\) are both non-compact because otherwise \(Y\) would be \(\hat{\theta}\)-nullbordant and thus support a psc metric by \cite[Theorem~1.5]{Ebert-Frenck:GLC_surgery}.
    We can now define the ends of $U_-$ to be $\E_-\hat{M}$ and the ends of $U_+$ to be $\E_+\hat{M}$ in order to turn \(\hat{M}\) into an open band.
    Then, by construction, the open band \(\hat{M}\) contains \(Y\) as a properly separating hypersurface.
    
    Finally, we need to verify \cref{pB}.
    To this end, assume by contradiction that $\Sigma \subset \hat{M}$ is a separating hypersurface that admits a psc metric.
    By \cref{lem:proper_separating}, we may assume without loss of generality that \(\Sigma\) is properly separating.
    Then it follows from \cref{lem:properly_separating_cobordism_class} that \(\Sigma\) and \(Y\) are \(\hat{\theta}\)-cobordant.
    But since the map \(Y \to \hat{B}\) is \(2\)-connected and \(\Sigma\) admits a psc metric, \cite[Theorem~1.5]{Ebert-Frenck:GLC_surgery} implies that \(Y\) also admits a psc metric, a contradiction.
\end{proof}

To describe the possible concrete applications of \cref{lem:tangential_structure}, it turns out to be enough to go through the possible tangential \(2\)-types of \(M\).
Recall that for an arbitrary \(n\)-manifold, its tangential \(2\)-type is the (unique up to homotopy) tangential structure \(\theta_M\) which factors the tangent bundle as \(M \to B_M \xrightarrow{\theta_M} \Bfree\Orth(n)\), where \(M \to B_M\) is \(2\)-connected and \(\theta_M \colon B_M \to \Bfree\Orth(n)\) is \(2\)-co-connected.
In other words, \(B_M\) is the second stage of the Moore--Postnikov tower for the map \(M \to \Bfree\Orth(n)\).
Now a simple diagram chase shows that if the hypotheses of \cref{lem:tangential_structure} are satisfied for some tangential structure \(\theta\) on \(M\), then they are already satisfied for \(\theta=\theta_M\).
Moreover, the tangential \(2\)-type can be described algebraically in terms of the fundamental group \(\pi = \pi_1 M\), the first Stiefel--Whitney class \(w \colon \pi \to \Z/2\), and an extension \(\hat{\pi} \twoheadrightarrow \pi\) determined by the second Stiefel--Whitney class whose kernel has order at most \(2\), see~\cite[\S 2]{Stolz98Concordance} for details.
However, for our purposes we do not need such a full description, and we only need to distinguish two cases depending on wether \(\pi_2 B_M \cong \Z/2\) or \(\pi_2 B_M = 0\) as the following proposition demonstrates.

\begin{prop}\label{prop:tangential_structure}
   Let \(n \geq 6\).
  Let \(M\) be a connected \(n\)-dimensional manifold without boundary and  \(Y \subset M\) a closed two-sided incompressible hypersurface that does not admit a psc metric.
  Suppose that one of the following conditions holds:
  \begin{myenuma}
    \item \(M\) is almost spin. 
    \item \(Y\) is totally non-spin. 
  \end{myenuma}
  Then there exists a covering \(\hat{M} \to M\) which is an open band with \cref{pB}.
\end{prop}
\begin{proof}
In light of the discussion in the previous paragraph, we need to check that in either case we can apply \cref{lem:tangential_structure} for the tangential structure \(\theta = \theta_M\) given by its \(2\)-type.
\begin{myenuma}
  \item If \(M\) is almost spin, that is, the universal covering \(\tilde{M}\) of \(M\) is spin, then the map \(M \to \Bfree\Orth(n)\) classifying the tangent bundle induces the zero map \(\pi_2 M \to \pi_2 \Bfree\Orth(n) = \Z/2\) because this map can be identified with the second Stiefel--Whitney class of \(\tilde{M}\)  via the Hurewicz isomorphism \(\HZ_2(\tilde{M}) \cong \pi_2(\tilde{M}) \cong \pi_2(M)\). Thus it follows that \(\pi_2 (B_M) = 0\) and so \cref{{lem:tangential_structure}} applies to every incompressible hypersurface \(Y \subset M\) which does not admit psc.
  \item On the other hand, if \(M\) is totally non-spin, that is \(\tilde{M}\) is non-spin, then by an analogous consideration involving the second Stiefel--Whitney class of \(\tilde{M}\) we necessarily have \(\pi_2 B_M \cong \Z/2\) and \(\pi_2 M \to \pi_2 B_M = \Z/2\) is surjective.
  Now in this situation the condition on the hypersurface \(Y \subset M\) in \cref{lem:tangential_structure} is that it is incompressible, itself totally non-spin, and does not admit psc. \qedhere
\end{myenuma}
\end{proof}

We are now ready to deduce \cref{thm:main_theorem}:

\begin{proof}[Proof of \cref{thm:main_theorem}]
  By \cref{prop:tangential_structure}, there exists a covering \(\hat{M}\) which is an open band with \cref{pB}.
  Now suppose that \(M\) admits a complete metric of non-negative scalar curvature \(g\).
  Let \(\hat{g}\) denote its lift to \(\hat{M}\).
  Then \cref{thm:analysis2} implies that \((\hat{M}, \hat{g})\) must be isometric to \((N \times \R, g_N + \D{t}^2)\) for a closed Ricci flat manifold \((N,g_N)\).
\end{proof}

We now turn to \cref{thm:npsc_theorem}. 
Note that a proof of \cref{thm:npsc_theorem} has already appeared recently in \cite[Theorem~1.1]{CPSZ21}, but we give a separate argument here because it also fits directly into our topological setup.
In fact, it is simpler than \cref{thm:main_theorem} as it does not need surgery arguments and instead only relies on \cref{lem:properly_separating_cobordism_class} together with some  homological considerations:

\begin{lem}\label{lem:npsc_implies_pB}
  Let \(M\) be an oriented open band and \(Y \subseteq M\) a properly separating hypersurface together with a map \(\phi \colon M \to Y_0\) to an \(\NPSC\) manifold \(Y_0\) such that the restriction \(\phi|_{Y} \colon Y \to Y_0\) has non-zero degree.
  Then \(M\) has \cref{pB}.
\end{lem}
\begin{proof}
We define a tangential structure \(M \xrightarrow{l} B \xrightarrow{\theta} \BO(n)\), where we set \(B = Y_0 \times \Bfree\SO(n)\) and \(l\) is induced by the map \(M \to Y_0\) together with the orientation of \(M\).
  Note that a \(\theta\)-structure on an \((n-1)\)-manifold \(N\) is the same as an orientation on \(N\) together with a map \(N \to Y_0\), and thus \(\Omega^\theta_{n-1} = \Omega_{n-1}^{\SO}(Y_0)\).
  In this picture, the degree of the map \(N \to Y_0\) can be read off from the transformation \(\Omega^{\SO}_{n-1}(Y_0) \to \HZ_{n-1}(Y_0; \Z) \cong \Z\).
  Suppose that \(\Sigma \subset M\) is a separating hypersurface, and assume without loss of generality that it is properly separating.
  Then, by \cref{lem:properly_separating_cobordism_class}, \(\Sigma \to Y_0\) and \(Y \to Y_0\) represent the same class of \(\Omega_{n-1}^{\SO}(Y_0)\).
  In particular, \(\deg(\Sigma \to Y_0) = \deg(Y \to Y_0) \neq 0\).
  Since \(Y_0\) is \(\NPSC\), this proves that \(\Sigma\) does not admit a psc metric and so \(M\) must have \cref{pB}.
\end{proof}

\begin{prop}\label{prop:npsc_implies_pB}
Let \(M\) be an oriented connected \(n\)-dimensional manifold and let \(\iota \colon Y \hookrightarrow M\) be a two-sided closed connected hypersurface that admits a map of non-zero degree \(\phi \colon Y \to Y_0\) to an aspherical \(\NPSC\) manifold \(Y_0\) and such that \(\ker(\pi_1 Y \xrightarrow{\iota_\ast} \pi_1 M) \subseteq \ker(\pi_1 Y \xrightarrow{\phi_\ast} \pi_1 Y_0)\).
  Then there exists a connected covering \(\hat{M} \to M\) that is an open band with \cref{pB}.
\end{prop}
\begin{proof}
  Let \(\Lambda \coloneqq \iota_\ast(\pi_1 Y) \subseteq \pi_1 M\) and let \(\hat{M} \to M\) be the covering with \(\pi_1 \hat{M} = \Lambda\).
  Then \(Y \hookrightarrow M\) lifts to an embedding \(Y \hookrightarrow \hat{M}\) which induces a surjection \(\pi_1 Y \twoheadrightarrow \Lambda\).
  By the assumption on the kernels of the induced maps on fundamental groups, it follows that the homomorphism \(\phi_\ast \colon \pi_1 Y \to \pi_1 Y_0\) factors as a composition \(\pi_1 Y \twoheadrightarrow \Lambda \to \pi_1 Y_0\).
  Since \(Y_0\) is aspherical and \(\pi_1 \hat{M} = \Lambda\), this implies that the map \(\phi \colon Y \to Y_0\) extends to a map \(\hat{M} \to Y_0\).
   As in the proof of \cref{lem:tangential_structure}, let \(U_\pm\) be the two components of \(\hat{M}\setminus Y\).
  Then \(U_\pm\) must be non-compact because otherwise \([Y \to Y_0] = 0 \in \Omega_{n-1}^{\SO}(Y_0)\) which would contradict the hypothesis \(\deg(Y \to Y_0) \neq 0\) (compare the proof of \cref{lem:npsc_implies_pB} above).
  Thus \(\hat{M}\) can be turned into an open band such that \(Y\) is a properly separating hypersurface.
  Thus the proposition follows from \cref{lem:npsc_implies_pB}.
\end{proof}

\begin{proof}[Proof of \cref{thm:npsc_theorem}]
Combine \cref{prop:npsc_implies_pB,thm:analysis2} analogously as in the proof of \cref{thm:main_theorem} above.
\end{proof}

\section{The codimension two obstruction}\label{sec:codimension_two}
In this section, we prove our codimension two obstruction results.
These are based on a reduction to a codimension one problem essentially following original ideas of \citeauthor{GromovLawson:PSCDiracComplete}~\cite[Theorem~7.5]{GromovLawson:PSCDiracComplete} and their adaptation by \citeauthor{HankePapeSchick:CodimensionTwoIndex}~\cite{HankePapeSchick:CodimensionTwoIndex}.
\begin{lem}[{cf.~\cites[Theorem~4.3]{HankePapeSchick:CodimensionTwoIndex}[Lemma~4.1.4]{Zeidler:SecondaryLargescaleIndex}}]\label{lem:extension_to_circle}
  Let \(X\) be a connected manifold without boundary.
  Let \(Y \subset X\) be a submanifold without boundary of codimension two whose normal bundle is trivial and suppose that the pair \((X,Y)\) is \(2\)-connected.
  Consider the manifold \(W \coloneqq X \setminus \mathcal{U}\) obtained by deleting a small open tubular neighborhood \(Y \times \Ball^2 \cong \mathcal{U} \subseteq X\).
  Then the map \(\partial W \cong Y \times \Sphere^1 \xrightarrow{\proj_2} \Sphere^1\) given by projection onto the second factor extends to a continuous map \(W \to \Sphere^1\).
  In particular, the map \(\pi_1 Y \times \Z = \pi_1(\partial W) \to \pi_1 W\) induced by the inclusion \(\partial W \hookrightarrow W\) is split-injective.
\end{lem}
\begin{proof}
  It suffices to show that the induced homomorphism \((\proj_2)_\ast \colon \pi_1 (\partial W) \to \pi_1(\Sphere^1) = \Z\) extends to a homomorphism \(\pi_1(W) \to \Z\).
  The hypotheses imply that the pair \((X, \mathcal{U})\) is also \(2\)-connected and so excision and the Hurewicz theorem show that \( \HZ_k(W, \partial W) \cong \HZ_k(X, \mathcal{U}) = 0\) for \(0 \leq k \leq 2\).
  In particular, the map \(\HZ_1(\partial W) \xrightarrow{\cong} \HZ_1(W)\) induced by the inclusion \(\partial W \hookrightarrow W\) is an isomorphism.
  The existence of the desired extension now follows from the following diagram because the map \((\proj_2)_\ast \colon \pi_1(\partial W) \to \pi_1(\Sphere^1) = \Z\) factors through the Hurewicz homomorphism.
  \[
  \begin{tikzcd}
    \pi_1(\partial W) \rar["\mathrm{hur}"] \dar & \HZ_1(\partial W) \dar["\cong"] \rar["(\proj_2)_\ast"] & \HZ_1(\Sphere^1) \rar[equal,"\mathrm{hur}"] & \pi_1(\Sphere^1) = \Z\\
    \pi_1(W) \rar["\mathrm{hur}"] &\HZ_1(W) \urar[dashed]
  \end{tikzcd}
  \]
  
  This also implies that the map \(\pi_1 Y \times \Z = \pi_1(\partial W) \to \pi_1 W\) is split injective because a retraction can be constructed using the map \(\pi_1 W \to \Z\) from the previous paragraph together with the map \(\pi_1 W \to \pi_1 X \cong \pi_1 Y\) induced by the inclusion \(W \hookrightarrow X\).
\end{proof}

\begin{prop}\label{prop:codim2_pB}
  Let \(n \geq 6\).
  Let \(X\) be a connected \(n\)-dimensional manifold and let \(Y \subset X\) be a closed connected submanifold of codimension two with trivial normal bundle such that the pair \((X,Y)\) is \(2\)-connected.
  Consider the manifold \(W \coloneqq X \setminus \mathcal{U}\) obtained by deleting a small open tubular neighborhood \(Y \times \Ball^2 \cong \mathcal{U} \subseteq X\).
  If \(Y \times \Sphere^1\) does not admit a metric of positive scalar curvature, then the double \(\double{W}\)of \(W\) is an open band with \cref{pB}.
\end{prop}
\begin{proof}
  We start with considering the tangential \(2\)-type \(X \to B_X \to \Bfree\Orth(n)\) of \(X\).
  By restriction to \(W \subset X\) this induces a tangential structure \(l' \colon W \to B_X\).
  Let \(p \colon W \to \Sphere^1\) be a map extending the projection \(\partial W \cong Y \times \Sphere^1 \to \Sphere^1\) which exists by \cref{lem:extension_to_circle}.
  We obtain a new tangential structure \(W \xrightarrow{l} B \coloneqq (B_X  \times \Sphere^1) \xrightarrow{\theta} \Bfree\Orth(n)\), where \(l = (l', p)\), and \(\theta\) is defined as projection to \(B_X\) followed by \(B_X \to \Bfree\Orth(n)\).
  Furthermore, this tangential structure \(l\) extends to the double \(\double{W}\) by reflection.
  Its restriction to the hypersurface \(\double{W} \supset \partial W \cong Y \times \Sphere^1\) is a \(2\)-connected map \(Y \times \Sphere^1  \to B = B_X \times \Sphere^1\) because by construction it is homotopic to \(l_{Y} \times \id_{\Sphere^1}\), where \(l_Y\) denotes the restriction of \(X \to B_X\) to \(Y \subset X\) and \(l_Y\) is \(2\)-connected.
  Thus \cref{lem:tangential_structure} (applied to the trivial covering) proves the desired conclusion. 
\end{proof}

\begin{proof}[Proof of \cref{thm:main_theorem_codim2}]
  Note that since \(Y\) does not admit positive scalar curvature and \(\dim(Y) = 5\), \cref{thm:main_theorem} implies that \(Y \times \Sphere^1\) does not admit positive scalar curvature either.
  Then we let \(X\) be the connected covering of \(M\) with \(\pi_1 X = \pi_1 Y\).
  It follows \(Y \hookrightarrow X\) and the pair \((X,Y)\) satisfies the hypotheses of \cref{prop:codim2_pB}.
  Thus, if we let \(W\) be as in the statement of \cref{prop:codim2_pB}, then its double \(\double{W}\) is an open band with \cref{pB}.
  Now assume by contradiction that \(M\) admits a complete metric of uniformly positive scalar curvature.
  Then by first lifting it to \(X\) and restricting it to \(W\), we obtain a complete metric of uniformly positive scalar curvature on \(W\).
  After changing this metric in a compact neighborhood of \(\partial W\), we obtain a smooth complete metric on \(\double W\) which has uniformly positive scalar curvature outside a compact subset, a contradiction to \cref{prop:no_upsc_outside_compact}.
\end{proof}

\begin{proof}[Proof of \cref{thm:npsc_theorem_codim2}]
 We first verify that \(Y_0 \times \Sphere^1\) is also \(\NPSC\).
 To this end, let \(N\) be an oriented \((n-1)\)-manifold and \(\phi \colon N \to Y_0 \times \Sphere^1\) a map of non-zero degree, where we assume \(N\) to be connected without loss of generality.
 Then let \(Z = \phi^{-1}(Y_0 \times \{\ast\})\) be a transversal pre-image.
 It follows that \(\iota \colon Z \hookrightarrow N\) and \(\phi \colon Z \to Y_0 \times \{\ast\} = Y_0\) satisfies the hypotheses of \cref{thm:npsc_theorem} and so \(N\) does not admit a psc metric.
 This proves that \(Y_0 \times \Sphere^1\) is \(\NPSC\).
 
 To prove the theorem, we again consider the connected covering of \(X \to M\) with \(\pi_1 X = \pi_1 Y\), and we let \(W\) be as in the statement of \cref{prop:codim2_pB}.
 By \cref{lem:extension_to_circle}, the map \(\pi_1 Y \times \Z = \pi_1(\partial W) \to \pi_1 W\) induced by the inclusion \(Y \times \Sphere^1 = \partial W \hookrightarrow W\) is injective and admits a retraction \(r \colon \pi_1 W \twoheadrightarrow \pi_1 Y \times \Z\).
 Since \(Y_0 \times \Sphere^1\) is aspherical, this implies that the map \(Y \times \Sphere^1 \to Y_0 \times \Sphere^1\) extends to a map \(W \to Y_0 \times \Sphere^1\) and subsequently to a map \(\double W \to Y_0 \times \Sphere^1\) on the double.
 In summary, \(\double W\) is an open band that contains \(Y \times \Sphere^1\) as a properly separating hypersurface and it satisfies the hypotheses of \cref{lem:npsc_implies_pB} because \(Y_0 \times \Sphere^1\) is \(\NPSC\).
 Thus \(\double W\) has \cref{pB} and the theorem now follows again from \cref{prop:no_upsc_outside_compact} as in the proof of \cref{thm:main_theorem_codim2} above.
\end{proof}

\section{Proof of the Partitioned Comparison Principle}\label{sec:proof_part_comparison}

We will prove Theorem \ref{thm:general} by contradiction.
We stress that all bands considered in this section are compact.
Under the assumption that $\wid(V_j,g)>\wid(M_j,g_{\varphi_j})$ for all $j\in\{1,\ldots, k+1\}$, we will produce a closed embedded hypersurface $\Sigma\subset X$ which separates $\p_-X$ and $\p_+X$ and admits a metric of positive scalar curvature.

This hypersurface $\Sigma$ will appear as the boundary of a $\mu$-bubble.
The key ingredient for the corresponding functional is the potential function $h \colon X\to\R$.
We use the ideas from \cite[Section 3]{Rae21} for each band $(V_j,g)$ and model space $(M_j,g_{\varphi_j})$ separately to produce $h_j \colon V_j\to\R$ as the concatenation of a strictly 1-Lipschitz band map $(V_j,g)\to(M_j,g_{\varphi_j})$ and the function $h_{\varphi_j}:M_j\to \R$. 
Subsequently, we use a gluing construction to paste all of the $h_j$ together to obtain a smooth function $h\colon X\to\R$ which is suitable for our purposes.

The idea to combine potential functions in this way, was already used in \cites{CeZeiPMT,ZeidlerWidthLargeness}.
The gluing construction is based on the following result:

\begin{lem}\label{lem:potential}
Let $h\colon [a,b]\to\R$ be a strictly monotonously decreasing smooth function such that
\[-\frac{n}{n-1}h^2-2h'=\sigma,\]
for some constant $\sigma\in\R$. 
Then, for every sufficiently small $\eps>0$, there exists a function $\hat{h} \colon [a-\eps,b+\eps]\to\R$ such that:
\begin{itemize}
\item[$\triangleright$] $\hat{h}(t)=h(t)$ for $t\in[a+\eps,b-\eps]$, 
\item[$\triangleright$] $\hat{h}(t)=h(a)$ in a neighborhood of $a-\eps$ and $\hat{h}(t)=h(b)$ in a neighborhood of $b+\eps$,
\item[$\triangleright$] $\hat{h}'\leq0$,
\item[$\triangleright$] $-\frac{n}{n-1}\hat{h}^2-2\hat{h}'\leq\sigma$ and $-\frac{n}{n-1}\hat{h}^2(t)-2\hat{h}'(t)<\sigma$ if $\hat{h}'(t)=0$.
\end{itemize}
\end{lem}

\begin{proof}
Let $\rho \colon \R\to[a,b]$ be a smooth function with:
\begin{itemize}
\item[$\triangleright$] $\rho(t)=a$ for $t\in(-\infty,a-\frac{\eps}{2}]$, $\rho(t)=t$ for $t\in[a+\frac{\eps}{2},b-\frac{\eps}{2}]$ and $\rho(t)=b$ for $t\in[b+\frac{\eps}{2},\infty)$.
\item[$\triangleright$] $0<\rho'(t)<1$ for $t\in(a-\frac{\eps}{2},a+\frac{\eps}{2})$ and $t\in(b-\frac{\eps}{2},b+\frac{\eps}{2})$.
\end{itemize} 
Then the function $\hat{h} \colon [a-\eps,b+\eps]\to\R$ defined by $\hat{h}=h\circ\rho$ has all of the desired properties.
The first two are immediate from the definition.
The third one holds since $\hat{h}'(t)=h'(\rho(t))\rho'(t)$ and $h'<0$ while $\rho'\geq0$.
To check the last property we point out that
\[-\frac{n}{n-1}\hat{h}^2(t)-2\hat{h}'(t)=-\frac{n}{n-1}h(\rho(t))-2h'(\rho(t))\rho'(t)=\sigma+2h'(\rho(t))(1-\rho'(t)).\]
Since $h'(\rho(t))<0$ and $0\leq\rho'\leq1$, the above is always $\leq \sigma$ and it is $<\sigma$ if $\rho'(t)<1$.
This holds true in particular when $\hat{h}'(t)=0$, that is, $\rho'(t)=0$.  
\end{proof}

In order to construct our functions $h_j\colon V_j\to\R$, we make use of the following basic existence result of band maps; for the proof, we refer to~\cites[Lemma~4.1]{Zhu:WidthEstimates}[Lemma~7.2]{Cecchini-Zeidler:ScalarMean}.

\begin{lem}\label{lem:widthmap}
Let $(V,g)$ be a Riemannian band and let $a<b$ be two real numbers. 
If $\wid(V,g)>b-a$, then there exists a smooth function $\beta \colon V\to[a,b]$ with $\beta(\p_-V)=a$, $\beta(\p_+V)=b$ and $\Lip(\beta)<1$.
\end{lem}

Next, for a Riemannian band $(X,g)$, we give conditions on the scalar curvature, the mean curvature of $\partial X$ and on the potential function $h$ so that the $\mu$-bubble associated to $h$ produces a separating closed hypersurface admitting a metric of positive scalar curvature.
For more details on $\mu$-bubbles, we refer the reader to~\cites[Section~5]{gromovFourLecturesScalar2019v6}{Zhu:WidthEstimates}.

\begin{prop}\label{prop:bubbles}
Let $n\leq 7$ and let $(X,g)$ be an $n$-dimensional oriented Riemannian band. Let $h\colon X\to\R$ be a smooth function with the property that 
\begin{equation}\label{eq:ConformalLaplacianCondition}
    \Sc(X,g)+\frac{n}{n-1}h^2-2|\nabla h|>0.
\end{equation}
Furthermore, suppose that the mean curvature satisfies
\begin{equation}\label{eq:BarrierCondition}
    \mean (\p_\pm X,g)>\pm h\bigr|_{\p_\pm X}.
\end{equation}
Then there exist a closed embedded hypersurface $\Sigma$ which separates $\p_-X$ and $\p_+X$ and a constant $b>0$ such that 
\begin{equation}\label{eq:ConformalLaplacianEstimate}
    \int_\Sigma\Bigl( |\nabla_\Sigma\psi|^2+\frac{1}{2}\Sc(\Sigma,g)\psi^2\Bigr) \,d\vol_\Sigma
    \geq b\int_\Sigma\psi^2\,d\vol_\Sigma,\qquad\forall\psi\in C^{\infty}(\Sigma).
\end{equation}
\end{prop}

\begin{proof}
Denote by $\mathcal{C}(X)$ the set of all Caccioppoli sets in $X$ which contain an open neighborhood of $\p_-X$ and are disjoint from $\p_+X$.
For $\hat{\Omega}\in\mathcal{C}(X)$ consider the functional
\begin{displaymath}
\mathcal{A}_h(\hat{\Omega})=\Ha^{n-1}(\p^*\hat{\Omega}\cap\mathring{X})-\int_{\hat{\Omega}}h\,d\Ha^n,
\end{displaymath}
where $\p^*\hat{\Omega}$ is the reduced boundary \cite[{Chapters~3,4}]{Giu84} of $\hat{\Omega}$. 
By Condition~\eqref{eq:BarrierCondition} and~\cite[{Lemma 4.2}]{Rae21}, there exists a smooth $\mu$-\emph{bubble} $\Omega\in\mathcal{C}(X)$, that is, a smooth Caccioppoli set with 
\[\mathcal{A}_h(\Omega)=\mathcal{I} \coloneqq \inf\{\mathcal{A}_h(\hat{\Omega})\bigr| \hat{\Omega}\in\mathcal{C}(X)\}.\]
Then $\Sigma\coloneqq\p \Omega\cap\mathring{X}$ is a closed embedded hypersurface that separates $\p_-X$ and $\p_+X$.
Let $\nu$ be the outward pointing unit normal vector field to $\Sigma$.
By the first variation formula for $\mathcal{A}_h$ (see \cite[{Lemma 4.3}]{Rae21}) the mean curvature of $\Sigma$ (computed with respect to $-\nu$) is equal to $h\bigr|_\Sigma$. By stability, from the second variation formula (see \cite[{Lemma 4.4}]{Rae21}) we deduce
\begin{align*}
\int_\Sigma \Bigl(2|\nabla_\Sigma\psi|^2+\Sc(\Sigma,g)\psi^2\Bigr) \,d\vol_\Sigma
\geq&\int_\Sigma\Bigl(\Sc(X,g) +\frac{n}{n-1}h^2+2g(\nabla_Xh,\nu)\Bigr)\psi^2\,d\vol_\Sigma\\
\geq&\min_{\Sigma}\Bigl\{\Sc(X,g)+\frac{n}{n-1}h^2-2|\nabla h|\Bigr\}\int_\Sigma\psi^2\,d\vol_\Sigma
\end{align*}
for all $\psi\in C^{\infty}(\Sigma)$.
By Condition~\eqref{eq:ConformalLaplacianCondition}, the previous inequality yields a constant $b>0$ such that Inequality~\eqref{eq:ConformalLaplacianEstimate} holds.
\end{proof}

We have gathered all the ingredients we need to prove Theorem \ref{thm:general}.

\begin{proof}[Proof of \cref{thm:general}]
Assume, by contradiction, that $\wid(V_j,g)>\wid(M_j,g_{\varphi_j})$ for all $j\in\{1,\ldots,k+1\}$.
Consider the functions \[h_{\varphi_j}(t)=(n-1)\frac{\varphi_j'(t)}{\varphi_j(t)} \colon [a_j,b_j]\to\R.\]
Since the $\varphi_j$ are strictly $\log$-concave, the functions $h_{\varphi_j}$ are strictly monotonously decreasing and the scalar curvature of $(M_j,g_{\varphi_j})$ is  given by
\begin{equation}\label{warpODE}
\sigma_j=\Sc(M_j,g_{\varphi_j})=-\frac{n}{n-1}h_{\varphi_j}^2-2h'_{\varphi_j}.
\end{equation}
For $j\in\{2,\ldots,k\}$, we apply Lemma \ref{lem:potential} to $h_{\varphi_j}$ and obtain smooth functions
\[\hat{h}_{\varphi_j} \colon [a_j-\eps,b_j+\eps]\to\R\]
with the aforementioned properties.

For $j=1$ we extend the domain of $h_{\varphi_j}$ to $[a_1-\eps, b_1]$ and apply the interpolation procedure of Lemma \ref{lem:potential} only on the right hand side of the interval to produce $\hat{h}_{\varphi_1} \colon [a_1-\eps,b_1+\eps]\to\R$.
For $j=k+1$ we extend the domain of $h_{\varphi_{k+1}}$ to $[a_{k+1},b_{k+1}+\eps]$ and apply the interpolation procedure of Lemma \ref{lem:potential} only on the left hand side of the interval to produce $\hat{h}_{\varphi_{k+1}} \colon [a_{k+1}-\eps,b_{k+1}+\eps]\to\R$.

By Lemma \ref{lem:widthmap}, there are smooth maps 
\[\beta_j \colon V_j\to[a_j-\eps,b_j+\eps]\] 
such that $\beta_j(\p_-V_j)=a_j-\eps$, $\beta_j(\p_+V_j)=b_j+\eps$ and $\Lip(\beta_j)<1$.
Define $h\colon X\to\R$ by $h(x)=\hat{h}_{\varphi_j}\circ\beta_j(x)$ if $x\in V_j$.
The function $h$ is continuous since 
\[\hat{h}_{\varphi_j}(b_j+\eps)=\mean(\p_+M_j,g_{\varphi_j})=-\mean(\p_-M_{j+1},g_{\varphi_{j+1}})=\hat{h}_{\varphi_{j+1}}(a_{j+1}-\eps)\]
for all $j\in\{1,\ldots,k\}$.
It is smooth since the $\hat{h}_{\varphi_j}\circ\phi_j$ are constant in a neighborhood of the separating hypersurfaces $\Sigma_i$, which partition the band.

Note that $h$ satisfies Condition~\eqref{eq:ConformalLaplacianCondition} by the chain rule, the fourth property of the $\hat{h}_{\varphi_j}$ from Lemma \ref{lem:potential}, and since $\Lip(\beta_j)<1$.
For the mean curvature of the boundary, the following holds true:
\[\mean(\p_-X,g)\geq \mean(\p_-M_1,g_{\varphi_1})=-\hat{h}_{\varphi_1}(a_1)>-\hat{h}_{\varphi_1}(a_1-\eps)=-h\bigr|_{\p_-X}\] 
and 
\[\mean(\p_+X,g)\geq \mean(\p_+M_{k+1},g_{\varphi_{k+1}})=\hat{h}_{\varphi_{k+1}}(b_{k+1})>\hat{h}_{\varphi_{k+1}}(b_{k+1}+\eps)=h\bigr|_{\p_+X}.\]
Hence, Condition~\eqref{eq:BarrierCondition} is satisfied as well.
By Proposition \ref{prop:bubbles}, there exist a closed embedded hypersurface $\Sigma$ which separates $\p_-X$ and $\p_+X$ and a constant $b>0$ such that Inequality~\eqref{eq:ConformalLaplacianEstimate} holds. 

If $n=2$, this yields an immediate contradiction by choosing $\psi=1$ in~\eqref{eq:ConformalLaplacianEstimate}. If $n=3$, we again choose $\psi=1$ and use Gau\ss-Bonnet to see that $\Sigma$ admits a psc metric.
If $n\geq4$, then $\Sigma$ admits a metric of positive scalar curvature by the conformal change argument of Schoen and Yau \cite{SchoenYau:HypersurfaceMethod}. This contradicts the fact that $X$ has Property \ref{pB}.
\end{proof}

\begin{proof}[Proof of \cref{thm:analysis}]
Consider the function 
\[\varphi_2 \colon \left(-\frac{\pi}{\sqrt{\kappa}n},\frac{\pi}{\sqrt{\kappa}n}\right)\to\R_+ \qquad t\mapsto\cos\left(\frac{\sqrt{\kappa}nt}{2}\right)^{\frac{2}{n}},\]
which is strictly $\log$-concave and has
\[h_{\varphi_2}(t)=-\sqrt{\kappa}(n-1)\tan\left(\frac{\sqrt{\kappa}nt}{2}\right).\]
Consider the function 
\[\varphi_1 \colon \R_+\to\R_+ \qquad t\mapsto t^{\frac{2}{n}},\]
which is strictly $\log$-concave and has
\[h_{\varphi_1}(t) \colon =\frac{2(n-1)}{nt}.\]
Since $h_{\varphi_1}(t)\to\infty$ as $t\to0$, there is a value $t_->0$ such that $H(\p_-X,g)\geq -h_{\varphi_1}(t_-)$.
By continuity, there are $\delta_1,\delta_2>0$ small enough such that
\[h_{\varphi_2}\left(\frac{-d+\delta_1}{2}\right)=h_{\varphi_1}(\ell+\delta_2),\]
while $\delta_2<t_-$ and hence $\ell+\delta_2-t_-<\ell$.
Let $(N,g_N)$ be a closed scalar flat Riemannian manifold.
We fix the model space
\[(M_1,g_{\varphi_1})=(N\times[t_-,\ell+\delta_2], \varphi_1^2(t)g_N+dt^2)\]
with scalar curvature equal to zero and $\wid<\ell$.

Let $\varphi_3 \colon \R_-\to\R_+$ be defined by $\varphi_3(t)=\varphi_1(-t)$. 
This function is strictly $\log$-concave and $h_{\varphi_3}(t)=-h_{\varphi_1}(-t)$. 
Since $h_{\varphi_3}(t)\to-\infty$ as $t\to0$, there is a value $t_+<0$ such that $H(\p_+X,g)\geq h_{\varphi_3}(t_+)$.
Similarly as before, we find $\delta_3,\delta_4>0$ such that
\[h_{\varphi_2}\left(\frac{d-\delta_3}{2}\right)=h_{\varphi_3}(-\ell-\delta_4),\]
while $\delta_4<-t_+$ and hence $\ell+\delta_4+t_+<\ell$.
We fix the model space
\[(M_3,g_{\varphi_3})=(N\times[t_-,\ell+\delta_4], \varphi_3^2(t)g_N+dt^2)\]
with scalar curvature equal to zero and $\wid<\ell$.

Finally, we fix the model space 
\[(M_2,g_{\varphi_2})=\left(N\times\left[\frac{-d+\delta_1}{2},\frac{d-\delta_3}{2}\right], \varphi_2^2(t)g_N+dt^2\right)\]
with scalar curvature equal to $\kappa n (n-1)$ and $\wid<d$.

It follows from Theorem \ref{thm:general} that $\wid(V_j,g)\leq\wid(M_j,g_{\varphi_j})$ for at least one $i\in\{1,2,3\}$.
Since $d=\wid(V_2,g)>\wid(M_2,g_{\varphi_2})$, we conclude that \[\min\{\wid(V_1,g),\wid(V_3,g)\}< \ell,\]
which is what we wanted to prove.
\end{proof}

\begin{proof}[Proof of \cref{thm:analysis1}]
Consider the function 
\[\varphi_2 \colon \left(-\frac{\pi}{\sqrt{\kappa}n},\frac{\pi}{\sqrt{\kappa}n}\right)\to\R_+ \qquad t\mapsto\cos\left(\frac{\sqrt{\kappa}nt}{2}\right)^{\frac{2}{n}},\]
which is strictly $\log$-concave and has
\[h_{\varphi_2}(t)=-\sqrt{\kappa}(n-1)\tan\left(\frac{\sqrt{\kappa}nt}{2}\right).\]
Consider the function 
\[\varphi_1 \colon \R_+\to\R_+ \qquad t\mapsto\sinh\left(\frac{\sqrt{\sigma n}t}{2\sqrt{n-1}}\right)^{\frac{2}{n}},\]
which is strictly $\log$-concave and has
\[h_{\varphi_1}(t) \coloneqq \sqrt{\frac{\sigma(n-1)}{n}}\coth\left(\frac{\sqrt{\sigma n}t}{2\sqrt{n-1}}\right).\]

Since $h_{\varphi_1}(t)\to\infty$ as $t\to0$, there is a value $t_->0$ such that the mean curvature $H(\p_-X,g)$ is greater or equal to $-h_{\varphi_1}(t_-)$.
By continuity, there are $\delta_1,\delta_2>0$ small enough such that $-\sigma>-\kappa n(n-1)\tan\left(\frac{\sqrt{\kappa}n(d-2\delta_1)}{4}\right)^2$ and 
\[h_{\varphi_2}\left(\frac{-d+\delta_1}{2}\right)=h_{\varphi_1}(\ell+\delta_2),\]
while $\delta_2<t_-$ and hence $\ell+\delta_2-t_-<\ell$.
Let $(N,g_N)$ be a closed scalar flat Riemannian manifold.
We fix the model space
\[(M_1,g_{\varphi_1})=(N\times[t_-,\ell+\delta_2], \varphi_1^2(t)g_N+dt^2)\]
with scalar curvature equal to $-\sigma$ and $\wid<\ell$.

Let $\varphi_3 \colon \R_-\to\R_+$ be defined by $\varphi_3(t)=\varphi_1(-t)$.
This function is strictly $\log$-concave and $h_{\varphi_3}(t)=-h_{\varphi_1}(-t)$. 
Since $h_{\varphi_3}(t)\to-\infty$ as $t\to0$, there is a value $t_+<0$ such that $H(\p_X,g)\geq h_{\varphi_3}(t_+)$.
Similarly as before, we find $\delta_3,\delta_4>0$ such that 
$-\sigma>-\kappa n(n-1)\tan\left(\frac{\sqrt{\kappa}n(d-2\delta_3)}{4}\right)^2$ and 
\[h_{\varphi_2}\left(\frac{d-\delta_3}{2}\right)=h_{\varphi_3}(-\ell-\delta_4),\]
while $\delta_4<-t_+$ and hence $\ell+\delta_4+t_+<\ell$.
We fix the model space
\[(M_3,g_{\varphi_3})=(N\times[-\ell-\delta_4,t_+], \varphi_3^2(t)g_N+dt^2)\]
with scalar curvature equal to $-\sigma$ and $\wid<\ell$.

Finally we fix the model space 
\[(M_2,g_{\varphi_2})=\left(N\times\left[\frac{-d+\delta_1}{2},\frac{d-\delta_3}{2}\right], \varphi_2^2(t)g_N+dt^2\right)\]
with scalar curvature equal to $\kappa n (n-1)$ and $\wid<d$.

It follows from Theorem \ref{thm:general} that $\wid(V_j,g)\leq\wid(M_j,g_{\varphi_j})$ for at least one $i\in\{1,2,3\}$.
Since $d=\wid(V_2,g)>\wid(M_2,g_{\varphi_2})$, we conclude that \[\min\{\wid(V_1,g),\wid(V_3,g)\}< \ell,\]
which is what we wanted to prove.
\end{proof}

{
    \hypersetup{hidelinks}
    \printbibliography
}

\end{document}